\documentclass{article}
\usepackage[utf8]{inputenc}
\usepackage{amsthm}
\usepackage{amsmath}
\usepackage{amsfonts}
\usepackage{color}
\usepackage{latexsym}
\usepackage{geometry}
\usepackage{enumerate}
\usepackage[shortlabels]{enumitem}
\usepackage{csquotes}
\usepackage{graphicx}
\usepackage{comment}
\usepackage{appendix}
\usepackage{hyperref}
\hypersetup{
    colorlinks=true,
    linkcolor=blue,
    filecolor=magenta,      
    urlcolor=cyan,
}
\usepackage{bm}
\usepackage{amssymb}

\emergencystretch=\maxdimen
\def \WW {\mathcal{W}}
\def \NN {\mathcal{N}}

\def \bW {\overline{\mathcal{W}}}
\def \dist {{\rm dist}}
\def \IR {\mathbb{R}}

\DeclareMathOperator{\Rm}{Rm}
\DeclareMathOperator{\Ric}{Ric}
\DeclareMathOperator{\Vol}{Vol}

\makeatletter
\newcommand*{\rom}[1]{\rm {\expandafter\@slowromancap\romannumeral #1@}}
\makeatother

\def\XXint#1#2#3{{\setbox0=\hbox{$#1{#2#3}{\int}$ }
\vcenter{\hbox{$#2#3$ }}\kern-.6\wd0}}

\protected\def\vts{%
  \ifmmode
    \mskip0.5\thinmuskip
  \else
    \ifhmode
      \kern0.08334em
    \fi
  \fi
}

\numberwithin{equation}{section}
\newtheorem{Theorem}{Theorem}[section]

\newtheorem{Lemma}[Theorem]{Lemma}
\newtheorem{Corollary}[Theorem]{Corollary}
\theoremstyle{definition}
\newtheorem{Definition}[Theorem]{Definition}

\def \WW {\mathcal{W}}

\def \NN {\mathcal{N}}

\def \bW {\overline{\mathcal{W}}}

\def \Rmin {R_{\operatorname{min}}}

\def \ve {\varepsilon}

\title{A local gap theorem for Ricci shrinkers}
\author{Pak-Yeung Chan, Zilu Ma, Yongjia Zhang}
\begin{document}

% \thanks{* University of California San Diego}

\maketitle
\begin{abstract}
  We prove a local gap theorem for Ricci shrinkers, which states that if the local $\mu$-functional at scale $1$ on a large ball centered at the minimum point of the potential function is close enough to $0$, then the shrinker must be the flat gaussian shrinker. In relation to our result, Yokota \cite{Yo09,Yo12} proved the same result assuming the global $\mu$-functional to be close enough to $0$. Our result shows an aspect of how the local geometry of a shrinker controls the global geometry, which is also discussed in \cite{LW19,LW20,LW21}.
\end{abstract}

\section{Introduction}
A shrinking gradient Ricci soliton, or Ricci shrinker for short, is a tuple $(M^n,g,f)$ of a smooth Riemannian manifold and a smooth function called the \emph{potential function}, satisfying the equation

\[
    \Ric + \nabla^2 f = \tfrac{1}{2}g.
\]
In our consideration of a Ricci shrinker, we normalize the function $f$ so that
\begin{align}\label{normalization}
        R + |\nabla f|^2 = f.
\end{align}
As an important subfield in the Ricci flow, the Ricci shrinker has been studied extensively by far, 
leading to many good results which have been facilitating the study of the Ricci flow in general, and its singularity formation in particular.

Surprisingly, Perelman \cite{Per02} shows that a shrinker is a critical point of his $\mathcal W$-functional. Because of the monotonicity of the $\mathcal W$-functional under the Ricci flow and the conjugate heat flow (namely, positive solutions to the conjugate heat equation), a shrinker limit can be obtained by proper scaling on the one hand (c.f. \cite{CZ11,MM15}), and a shrinker and its potential function must be strongly related to the conjugate heat equation and (logarithmic) Sobolev inequalities on the other hand (c.f. \cite{CN09,LW20}); the latter shall be our chief viewpoint in this paper.

Indeed, the potential function $f$ can be regarded as a singular conjugate heat kernel on the canonical form of the shrinker, which, due to Cao-Zhou \cite{CZ10}, always admits  gaussian upper and lower estimates. Precisely,
\begin{align}\label{Cao-Zhou}
        \frac{1}{4} \Big( \dist_g(x,o) - 5n   \Big)_+^2
    \le f(x)
    \le \frac{1}{4} \Big( \dist_g(x,o) + \sqrt{2n}\,\Big)^2,
\end{align}
where $o$ is a fixed \emph{minimum point} of $f.$

Another important geometric quantity of the shrinker is the $f$-volume $\displaystyle \int_M e^{-f}dg$, which also reveals many of the shrinker's properties. For instance, by implementing the $f$-volume, Wylie \cite{W08} proved that the fundamental group of a shrinker must be finite. Since the $f$-volume of a shrinker is always finite \cite{CZ10}, we may define the quantity $\mu_g$ as its logarithm, namely, 
\[
    \int_M (4\pi)^{-\frac{n}{2}}e^{-f}dg = e^{\mu_g}.
\]
The significance of $\mu_g$ is sufficently explored by Carrillo-Ni \cite{CN09} and Li-Wang \cite{LW20}. Indeed, if we let 
\begin{align}\label{new normalization}
\tilde f:=f+\mu_g,
\end{align}
then 
\begin{align}\label{new normalization of f}
\int(4\pi)^{-\frac{n}{2}}e^{-\tilde f}dg=1,
\end{align}
 and $\tilde f$ is exactly the minimizer of Perelman's functional $\mathcal W(g,\cdot,1)$. In fact, much more can be said:
\begin{gather}\label{shrinker entropy}
    2\Delta \tilde f-|\nabla \tilde f|^2+R+\tilde f-n\equiv \mu_g,
    \\\nonumber
    \mu_g=\mathcal W(g,\tilde f,1)=\mu(g,1)=\nu(g).
\end{gather}
The definitions of the $\mathcal W$, $\mu$, and $\nu$ functionals are found in \cite{Per02}. For this reason, the quantity $\mu_g$ is often called the \emph{shrinker entropy}.

By applying his techniques of estimating the reduced volume, Yokota \cite{Yo09,Yo12} proved a gap theorem showing how the  shrinker entropy is involved with the global geometry.

\begin{Theorem}[Yokota's gap theorem \cite{Yo12}]\label{Yokota gap}
There is a positive number $\varepsilon=\varepsilon(n)>0$ depending only on the dimension $n$ with the following property. Let $(M^n,g,f)$ be a Ricci shrinker and let $\mu_g$ be the shrinker entropy. Assume that $\mu_g\ge-\varepsilon$, then $(M^n,g,f)$ must be the flat gaussian shrinker.
\end{Theorem}

    The above theorem of Yokota states that the $\mu_g$ of a nonflat shrinker cannot be too close to zero, which is tantamount to saying that, if the global logarithmic Sobolev constant is close enough to zero, then there must be a global control of the curvature, that is, the curvature must be zero. Its local version, namely, the local logarithmic Sobolev constant controlling the local geometry, was first developed by Perelman in his famous pseudolocality theorem \cite{Per02}. However, on a special object such as a Ricci shrinker, the local geometry can sometimes control the global geometry; this is seen, for instance, in \cite{LW19,LW20,LW21}. The theorem which we prove in this paper is of this sort.

We shall consider the local version of Perelman's $\mu$-functional on a smooth manifold $M^n$, which was studied in \cite{W18}. For any open set $\Omega\subset M$ and any smooth metric $g$, define
\begin{align*}
    \mu(\Omega,g,\tau)&:=\inf\left\{\bW(g,u,\tau)\,\bigg|\, \int_M u^2\,dg=1,\ \ u\in C_0^{0,1}(\Omega)\right\},
    \\
    \ \ 
    \\
    \nu(\Omega,g,\tau)&:=\inf\left\{\mu(\Omega,g,s)\, \big|\, s\in(0,\tau]\right\},
\end{align*}
where
\begin{align*}
    \bW(g,u,\tau):=\int_M\left(\tau\big(4|\nabla u^2|+Ru^2\big)-u^2\log u^2\right)\,dg-n-\frac{n}{2}\log4\pi\tau.
\end{align*}
It is well-known that $\mu(\Omega,g,\tau)$ is a local logarithmic Sobolev constant. So long as $\mu(\Omega,g,\tau)>-\infty$, there is a logarithmic Sobolev inequality on $\Omega$, namely,
\begin{align}\label{log Sobolev}
    \int_M u^2\log u^2\, dg +n+\frac{n}{2}\log 4\pi\tau +\mu(\Omega,g,\tau)\le \int_M\tau\big(4|\nabla u|^2+Ru^2\big)\,dg
\end{align}
for any $u\in C_0^{0,1}(\Omega)$ satisfying $\displaystyle\int_Mu^2=1$.

 If the Riemannian manifold $(M,g)$ is complete and noncompact, and if $\Omega$ is precompact, then the functional $\mu(\Omega,g,\tau)$ reflects only the geometric property within $\Omega$; it is plausible that anything could happen near the spatial infinity of the manifold while $\mu(\Omega,g,\tau)$ remains close to zero. However, in the next theorem we shall prove that this observation is not true on a shrinker. If the geometry in a large ball is good enough (in the sense that the local $\mu$ functional is close enough to $0$), then the whole shrinker must be Euclidean.

\begin{Theorem}
\label{thm: local gap}
   For any dimension $n$ there is a positive constant $\delta(n)>0$ with the following property. Let $(M^n,g,f)$ be a complete Ricci shrinker. Let $o\in M$ be a minimum point of $f$. If
    \[
        \mu(B(o,1/\delta), g,1) \ge -\delta,
    \]
then $(M^n,g,f)$ is the flat gaussian shrinker.
\end{Theorem}

It is worth noting that the local $\mu$-functional is not a priori involved with curvature. Presumably, a manifold could have very good isoperimetic constant (and hence very nice logarithmic Sobolev constant) with arbitrarily large curvature. To make this point clear, we present the following corollary of Theorem \ref{thm: local gap}.

\begin{Corollary}\label{coro}
For any dimension $n$ there is a positive constant $\delta(n)>0$ with the following property. Let $(M^n,g,f)$ be a complete Ricci shrinker. Let $o\in M$ be a minimum point of $f$. If %the isoperimetric constant of $B(o,1/\delta)$ is $\delta$-close to that of the Euclidean space, namely,
\begin{align}
\left(\operatorname{Area}_{g}(\partial \Omega)\right)^{n} \geq(1-\delta) n^n\omega_{n}\left(\operatorname{Vol}_{g}(\Omega)\right)^{n-1}\qquad\text{ for any regular } \Omega\subset B(o,1/\delta),
\end{align}
then $(M^n,g,f)$ is the flat gaussian shrinker.
\end{Corollary}

As an application of the above theorem and corollary, we prove another local gap theorem for Ricci shrinkers. To state our result, we review the notion of curvature scale.

\begin{Definition}\label{def_radius}
    Let $(M^n,g)$ be a Riemannian manifold, the curvature scale at $x\in M$ is defined as
    \begin{align}\label{def_radius1}
        r_{\Rm}(x):=\sup\big\{r>0\,\big|\, |{\Rm}|\le r^{-2}\ \text{ on }\ B(x,r)\big\}.
    \end{align}
    Similarly, if $(M^n,g_t)_{t\in I}$ is a Ricci flow, then the curvature scale at $(x,t)\in M\times I$ is defined as
    \begin{align}\label{def_radius2}
        r_{\Rm}(x,t):=\sup\big\{r>0\,\big|\, |{\Rm}|\le r^{-2}\ \text{ on }\ B_{g_t}(x,r)\times[t-r^2,t]\big\}.
    \end{align}
\end{Definition}

\begin{Theorem}\label{rm gap}
    For any dimension $n$ there is a positive number $\varepsilon(n)>0$ with the following property. Let $(M^n,g,f)$ be a complete shrinker. Let $o\in M$ be a minimum point of $f$. If $r_{\Rm}(o)\ge \varepsilon^{-1}$, then $(M^n,g,f)$ is the flat gaussian shrinker.
\end{Theorem}

To conclude the introduction, we point out that Theorem \ref{thm: local gap}, Corollary \ref{coro}, and Theorem \ref{rm gap} are in the spirit of \cite[Corollary 7]{LW20}. Corollary \ref{coro} can even be directly derived from the techniques in \cite{LW20} (see Section 3 below). However, Theorem \ref{thm: local gap} does not follow from \cite[Corollary 7]{LW20} as a straightforward corollary. While the latter highly depends on \cite{LLW21}, our approach is totally different. Other gap theorems for shrinkers under global conditions are also seen in \cite{CCL22, MW11, Zh18, Zh20}.

\section{The proof of Theorem \ref{thm: local gap}}
In this section, we present the proof of Theorem \ref{thm: local gap}. We first show a coarse uniform estimate for the local and global entropies using the volume estimates in \cite{LLW21}. The shrinker entropy estimate is then refined by localizing its minimizer $\tilde f$. The refined estimate implies that $\mu_g$ is close to $0$ provided the local $\mu$ entropy is sufficiently small. 
The argument is similar to the proof of the local monotonicity formula for the Ricci flow, see \cite{W20, TZ21, CMZ21}.
Theorem \ref{thm: local gap} then follows from Theorem \ref{Yokota gap} (see also \cite{Yo09, Yo12}). 

\begin{Lemma}
 Let $(M^n,g,f)$ be a Ricci shrinker normalized as in \eqref{normalization}. Let $o$ be a minimum point of $f$. Then we have
\begin{align*}
    \Vol_g\left(B(o,1)\right)\ge c(n)\exp\left(\mu\big(B(o,1),g,1\big)\right),
\end{align*}
where $c(n)$ is a dimensional constant. 
\end{Lemma}

\begin{proof}
The proof is a slight modification of \cite[Lemma 2.4]{LLW21}, we include it here for the convenience of the reader.

Applying \cite[Theorem 1.2]{WW09} to the concentric balls $B(o,1)$ and $B(o,1/2)$, we have 
\begin{align}\label{WWcomparison}
    \frac{\int_{B(o,1)}e^{-f}dg}{\int_{B(o,1/2)}e^{-f}dg}\le C(n)\exp\left(\sup_{B(o,1)}|\nabla f|\right),
\end{align}
where $C(n)$ is a dimensional constant. By \eqref{normalization}, \eqref{Cao-Zhou}, and the fact that $R\ge 0$ (c.f. \cite{Che09}), we also have that
\begin{align}\label{boundednessoff}
    0\le f(x)\le n,\qquad |\nabla f|\le \sqrt{n},\qquad \text{ for all }\qquad x\in B(o,1).
\end{align}
It follows from \eqref{WWcomparison} that
\begin{align}
    \Vol_g\left(B(o,1/2)\right)\ge c(n)\Vol_g\left(B(o,1)\right).
\end{align}

Now, we consider a smooth cutoff function $\eta$, which is compactly supported in $B(o,1)$ with
\begin{align*}
    0\le \eta\le 1, \qquad \eta\big|_{B(o,1/2)}\equiv 1,\qquad |\nabla \eta|\le 4.
\end{align*}
Letting
$$L:=\int_M \eta^2\,dg,\qquad \varphi=L^{-\frac{1}{2}}\eta,$$
then we have 
\begin{gather}\label{volumecomparison}
   \varphi\in C^{0,1}_0(B(o,1)),\qquad  \int_M\varphi^2=1,\\\nonumber \Vol_g\left(B(o,1)\right)\ge L\ge \Vol_g\left(B(o,1/2)\right)\ge c(n)\Vol_g\left(B(o,1)\right).
\end{gather}
Applying the logarithmic Sobolev inequality \eqref{log Sobolev} to $\varphi$ with $\tau=1$ and $\Omega=B(o,1)$, we have
\begin{align*}
    n+\frac{n}{2}\log4\pi+\mu\big(B(o,1),g,1\big)&\le \left(\frac{64}{L}\Vol_g\left(B(o,1)\right)+\sup_{B(o,1)}R\right)-\int_M\varphi^2\log\varphi^2\, dg
    \\
    &\le \left(\frac{64\Vol_g\left(B(o,1)\right)}{c(n)\Vol_g\left(B(o,1)\right)}+n\right)+\log L-\frac{1}{L}\int_{M}\eta^2\log\eta^2\,dg
    \\
    &\le C(n)+\log L+\frac{1}{eL}\Vol_g\left(B(o,1)\right)
    \\
    &\le \log L+C(n),
\end{align*}
where we have applied \eqref{boundednessoff}, \eqref{volumecomparison}, and the fact that $R\le f$; this finishes the proof of the lemma.

\end{proof}

\begin{Lemma}
\label{lem: rough local mu bound}
 Let $(M^n,g,f)$ be a Ricci shrinker normalized as in \eqref{normalization}. Let $o$ be a minimum point of $f$. Then we have
\[
    \mu_g \ge \mu(B(o,1),g,1) - C(n),
\]
where $C(n)$ is a dimensional constant.
\end{Lemma}
\begin{proof}
 \cite[Lemma 2.5]{LLW21} shows that  
    \[
        \Vol_g\left(B(o,1)\right) \le C(n) e^{\mu_g}.
    \]
In combination with the above lemma, we complete the proof.
\end{proof}

\begin{proof}[{Proof of Theorem \ref{thm: local gap}}]

Let $\delta\in (0,1)$ be a constant to be determined. Assume $(M^n,g,f)$ is a Ricci shrinker normalized as in \eqref{normalization}, satisfying 
\begin{align*}
    \mu\big(B(o,1/\delta),g,1\big)\ge-\delta,
\end{align*}
where $o$ is the minimum point of $f$.

Next, we shall use the function $(4\pi)^{-\frac{n}{2}}e^{-\tilde f}$ to create a test function that can be properly applied to the logarithmic Sobolev inequality determined by $\mu\big(B(o,1/\delta),g,1\big)$, where $\tilde f$ is defined in \eqref{new normalization}, normalized in the way that $$\int_M (4\pi)^{-\frac{n}{2}}e^{-\tilde f}\,dg=1.$$
Two things are to be attended to in the construction of the test function. One is that the function should have unit $L^2$-norm, the other is that the function should be compactly supported in $B(o,1/\delta)$. Let us deal with these two points one by one.

By a volume estimate proved by Munteanu-Wang \cite[Theorem 1.4]{MW14},
\begin{align}\label{MW14}
       \Vol_g\left(B(o,r)\right) \le C_n r^n\qquad\text{ for all }\qquad r>0.
\end{align}
It follows that, for any $r\ge 40n,$
\begin{align}\label{outter smallness}
        \int_{M\setminus B(o,r)}(4\pi)^{-\frac{n}{2}}e^{-\tilde f}\,dg
    &\le C_n \int_{M\setminus B(o,r)}
    %e^{-\frac{1}{4}(r-5n)_+^2 - \mu_g}\, dg
\exp\left(-\tfrac{1}{4}\big(\dist(x,o)-5n\big)_+^2 - \mu_g\right)\, dg
    \\\nonumber
    &\le C_n e^{-\mu_g - \frac{1}{5}r^2}
    \le C_n e^{- \frac{1}{5}r^2},
\end{align}
where in the last inequality above we applied Lemma \ref{lem: rough local mu bound} and the fact that $\mu\big(B(o,1),g,1\big)\ge \mu\big(B(o,1/\delta),g,1\big)\ge -\delta\ge -1$.

Let $\eta$ be a standard cutoff function on $\IR$ such that
\[
    \eta|_{(-\infty,1/2]}=1,\quad
    \eta|_{[1,\infty)}=0,\quad
    - 3\sqrt{\eta} \le \eta'\le 0.
\]
Let
\[
    \phi(x) = \eta\left(\tfrac{\dist(x,o)}{\delta^{-1}}\right).
\]
Then we obviously have that
\begin{align}\label{cutoff}
\phi\in C_{0}^1\big(B(o,1/\delta)\big)\qquad \text{ and }\qquad |\nabla \phi|\le 3\delta.
\end{align}
By \eqref{outter smallness}, if we take $\delta\le \bar\delta(n)$, then we have
\begin{align}\label{lower bound of V}
     V
    &:= \int_M(4\pi)^{-\frac{n}{2}}
    \phi^2 e^{-\tilde f}\, dg 
    \ge \int_{ B\big(o,\frac{1}{2}\delta^{-1}\big)}(4\pi)^{-\frac{n}{2}} e^{-\tilde f}\, dg
    \\\nonumber
    &=1-\int_{ M\setminus B\big(o,\frac{1}{2}\delta^{-1}\big)}(4\pi)^{-\frac{n}{2}} e^{-\tilde f}\, dg
    \ge 1 - C_n e^{-\frac{1}{20}\delta^{-2}}
    \ge \tfrac{1}{2}.
\end{align}
We shall use 
\[
     u := \phi\sqrt{(4\pi)^{-\frac{n}{2}}
     e^{-\tilde f}/V}
\]
as the test function. Since $\mu(B(o,1/\delta),1)\ge -\delta$ and $\displaystyle\int_M u^2\,dg=1$, applying \eqref{log Sobolev} to $u$ with $\tau=1$ and $\mu(\Omega,g,\tau)=-\delta$, we have

\begin{align}\label{last but two}
     -\delta 
    &\le \int \left(4|\nabla u|^2
        + Ru^2
    \right)\, dg-\int_M u^2\log u^2\,dg
    - \tfrac{n}{2}\ln (4\pi) - n\\\nonumber
    &= 
    \int \left(4\left|\nabla \log \phi-\tfrac{1}{2}\nabla\tilde f\right|^2u^2
        + Ru^2\right)\,dg-\int_M\left(\log \phi^2-\tfrac{n}{2}\log 4\pi -\tilde f-\log V\right)u^2\,dg - \tfrac{n}{2}\ln (4\pi) - n
        \\\nonumber
&=\frac{1}{V}\int_M\left(4\left|\nabla\phi\right|^2-4\phi\left\langle\nabla\phi,\nabla\tilde f\right\rangle\right)(4\pi)^{-\frac{n}{2}}e^{-\tilde f}\,dg
%\int_M\left(4\left|\nabla\phi\right|^2-4\phi\left\langle\nabla\phi,\nabla\tilde f\right\rangle\right)(4\pi)^{-\frac{n}{2}}e^{-\tilde f}\,dg
+\int_M\left(\left|\nabla\tilde f\right|^2+R+\tilde f-n\right)u^2\,dg+\log V
\\\nonumber
&\qquad-\frac{1}{V}\int_M\left(\phi^2\log\phi^2\right) (4\pi)^{-\frac{n}{2}}e^{-\tilde f}\,dg       
   \\\nonumber
   &=\frac{1}{V}\int_M\left(4\left|\nabla\phi\right|^2+2\phi^2\Delta\tilde f-2\phi^2\left|\nabla\tilde f\right|^2\right)(4\pi)^{-\frac{n}{2}}e^{-\tilde f}\,dg
   %\int_M\left(4\left|\nabla\phi\right|^2+2\phi^2\Delta\tilde f-2\phi^2\left|\nabla\tilde f\right|^2\right)(4\pi)^{-\frac{n}{2}}e^{-\tilde f}\,dg
   +\int_M\left(\left|\nabla\tilde f\right|^2+R+\tilde f-n\right)u^2\,dg+\log V
\\\nonumber
&\qquad-\frac{1}{V}\int_M\left(\phi^2\log\phi^2\right) (4\pi)^{-\frac{n}{2}}e^{-\tilde f}\,dg       
   \\\nonumber
&=\int_M\left(2\Delta \tilde f-\left|\nabla\tilde f\right|^2+R+\tilde f-n\right)u^2\,dg+\log V
\\\nonumber
&\qquad +\frac{4}{V}\int_M|\nabla \phi|^2(4\pi)^{-\frac{n}{2}}e^{-\tilde f}\,dg
%+4\int_M|\nabla \phi|^2(4\pi)^{-\frac{n}{2}}e^{-\tilde f}\,dg
-\frac{1}{V}\int_M\left(\phi^2\log\phi^2\right) (4\pi)^{-\frac{n}{2}}e^{-\tilde f}\,dg 
\\\nonumber
&=\mu_g+\log V+\frac{4}{V}\int_M|\nabla \phi|^2(4\pi)^{-\frac{n}{2}}e^{-\tilde f}\,dg
%+4\int_M|\nabla \phi|^2(4\pi)^{-\frac{n}{2}}e^{-\tilde f}\,dg
-\frac{1}{V}\int_M\left(\phi^2\log\phi^2\right) (4\pi)^{-\frac{n}{2}}e^{-\tilde f}\,dg, 
\end{align}
where $\mu_g$ is the constant defined in \eqref{shrinker entropy}. For the last two terms in the above inequality, we first apply \eqref{cutoff} and \eqref{lower bound of V}
to obtain
\begin{align}\label{last but one}
\frac{4}{V}\int_M|\nabla \phi|^2(4\pi)^{-\frac{n}{2}}e^{-\tilde f}\,dg\le 72\delta^2\int (4\pi)^{-\frac{n}{2}}e^{-\tilde f}\,dg=72\delta^2.
%4\int_M|\nabla \phi|^2(4\pi)^{-\frac{n}{2}}e^{-\tilde f}\,dg\le 36\delta^2\int (4\pi)^{-\frac{n}{2}}e^{-\tilde f}\,dg=36\delta^2.
\end{align}
Secondly, since, by \eqref{new normalization of f}, $d\nu:=(4\pi)^{-\frac{n}{2}}e^{-\tilde f}\,dg$ is a probability measure, then, by Jensen's inequality, we have
\begin{align}\label{last}
\frac{1}{V}\int_M\left(\phi^2\log\phi^2\right) (4\pi)^{-\frac{n}{2}}e^{-\tilde f}\,dg&=\frac{1}{V}\int_M\phi^2\log\phi^2 d\nu\ge \frac{1}{V}\left(\int_M\phi^2\,d\nu\right)\log\left(\int_M\phi^2\,d\nu\right)
\\\nonumber
&=\frac{1}{V}\cdot V\log V
\\\nonumber
&=\log V.
\end{align}
Combining \eqref{last but two}, \eqref{last but one}, and \eqref{last}, we have
\begin{align}
\mu_g\ge -\delta-72\delta^2.
%\mu_g\ge -\delta-36\delta^2.
\end{align}
Finally, taking $\delta\le\overline\delta(n)$, the conclusion of our theorem follows from Theorem \ref{Yokota gap}.
\end{proof}

\section{A pseudolocality point of view}
The proof of Theorem \ref{thm: local gap} in the previous  section, though very concise, does not reflect much of the Ricci flow mechanism developed by Hamilton, Perelman, and others.  Indeed, a Ricci shrinker $(M^n,g,f)$ generates an ancient Ricci flow  in the following way. Let
\begin{align}\label{canonical form}
\tau_t:=1-t,\qquad \frac{d}{dt}\phi_t=\frac{1}{\tau_t}\nabla_gf\circ\phi_t,\qquad \phi_0=\operatorname{id},\qquad g_t=\tau_t\phi_t^*g.
\end{align}
Then $(M,g_t)_{t\in(-\infty,1)}$ satisfies the Ricci flow equation, and is called the \emph{canonical form} of the Ricci shrinker. If we let $\tilde f_t=f\circ\phi_t+\mu_g$, then
\begin{align}\label{standard CHK}
\tilde u_t=(4\pi\tau_t)^{-\frac{n}{2}}e^{-\tilde f_t}
\end{align}
is a solution to the conjugate heat equation. In the current and following section, we shall view a Ricci shrinker dynamically as a Ricci flow by considering its canonical form, and show how the proofs of Theorem \ref{thm: local gap} and Corollary \ref{coro} is related to the classical theory of Ricci flow. First of all, we recall the pseudolocality theorem of Li-Wang, which is an adaptation of Perelman's pseudolocality theorem (c.f. \cite{W20,LW20}) to the canonical form of a Ricci shrinker.

\begin{Theorem}[Pseudolocality {\cite[Theorem 24]{LW20}}]\label{pseudolocality}
There exist positive numbers $\varepsilon_0(n)>0$ and $\delta_0(n)>0$ with the following properties. Let $(M,g_t)_{t\in(-\infty,1)}$ be the canonical form  of a Ricci shrinker. Suppose $t_0\in(-\infty,1)$ and $B_{g_{t_0}}(x_0,r)\subset M$ is a geodesic ball satisfying
$$\nu\big(B_{g_0}(x_0,r),g_0,r^2\big)>-\delta_0.$$
Then for each $t\in(t_0,\min\{t_0+\varepsilon_0r^2,1\})$ and $x\in B_{g_t}(x_0,\frac{1}{2}r)$, we have
\begin{gather*}
|{\Rm}|(x,t)\le (t-t_0)^{-1},\\
\inf_{\rho\in(0,\sqrt{t-t_0}\, ]} \rho^{-n}\Vol_{g_t}\left(B_{g_t}(y,\rho)\right)\ge \tfrac{1}{2}\omega_n.
\end{gather*}
\end{Theorem}

We shall show how Theorem \ref{pseudolocality} implies Corollary \ref{coro}. However, it is worth pointing out that Theorem \ref{thm: local gap} itself cannot be immediately reduced to a corollary of Theorem \ref{pseudolocality}. One can nevertheless apply Theorem \ref{pseudolocality} to prove a weaker version of Theorem \ref{thm: local gap}, namely, with the $\mu\big(B(o,1/\delta),g,1\big)>-\delta$ condition replaced by $\nu\big(B(o,1/\delta),g,1\big)>-\delta$. We recall the following technical lemma due to Wang \cite[Lemma 3.5]{W18}, %[Lemma 3.5 Wang local entropy part A], 
which, in combination with Theorem \ref{thm: local gap}, already leads to a proof of Corollary \ref{coro}.

\begin{Lemma} \emph{\cite[Lemma 3.5]{W18}}%[Lemma 3.5 in Wang]
\label{nu estimate}
Let $(M^n,g)$ be a Riemannian manifold with nonnegative scalar curvature. Assume that there are $o\in M$ and $\delta\in(0,1)$ such that
\begin{align}\label{isoperimetric}
\left(\operatorname{Area}_{g}(\partial \Omega)\right)^{n} \geq(1-\delta) n^n\omega_{n}\left(\operatorname{Vol}_{g}(\Omega)\right)^{n-1}\qquad\text{ for any regular } \Omega\subset B(o,1/\delta).
\end{align}
Then we have
\begin{align*}
\nu(B(o,1/\delta),g,1/\delta^2)\ge n\log(1-\delta).
\end{align*}
\end{Lemma}

\begin{proof}[A proof of Corollary \ref{coro} via Theorem \ref{pseudolocality}]
Let $(M^n,g,f)$ be the Ricci shrinker in the statement of Corollary \ref{coro} and let $(M,g_t)_{t\in(-\infty,1)}$ be its canonical form. We take $\delta\le \overline{\delta}(n)$, such that
\begin{align*}
\varepsilon_0/\delta^2\ge 1,\qquad n\log(1-\delta)\ge -\delta_0,
\end{align*}
where $\varepsilon_0$ and $\delta_0$ are the constants in Theorem \ref{pseudolocality}. Due to Lemma \ref{nu estimate}, we can apply Theorem \ref{pseudolocality} to $(M,g_t)_{t\in(-\infty,1)}$ at $(o,0)$ with $r=\delta^{-1}$. This leads to 
\begin{align}\label{upper bound due to pseudolocality}
R(o,t)\le \frac{1}{t},\qquad \text{ for all }t\in(0,1).
\end{align}

On the other hand, since $o$ is the minimum point of the potential function $f$, we have that $o$ is a static point under the flow $\phi_t$ which defines the canonical form \eqref{canonical form}. Thus, we have
\begin{align}\label{lower bound due to canonical form}
R(o,t)=\frac{1}{\tau_t} R_g(\phi_t(o))=\frac{1}{1-t}R_g(o),\qquad \text{ for all }t\in(-\infty,1).
\end{align}
Combining \eqref{upper bound due to pseudolocality} and \eqref{lower bound due to canonical form}, we have
\begin{align*}
R_g(o)\le \frac{1-t}{t},\qquad\text{ for all } t\in(0,1).
\end{align*}
Taking $t\to 1$, we have
$$R_g(o)=0.$$
Due to \cite{Che09} and the strong minimumm principle, the Ricci shrinker must be Ricci flat, and the shrinker equation $$\nabla^2f=\tfrac{1}{2}g$$ 
immediately implies that the shrinker is gaussian (c.f. \cite[Theorem 1]{PRS11}).
\end{proof}

\section{A local monotonicity point of view}

In this section, we shall present a proof of Theorem \ref{thm: local gap} by applying Wang's local monotonicity technique \cite{W18} sharpened by the authors \cite{CMZ21}. To avoid too much technical verbosity, to focus on the central idea, and since we have already shown a rigorous proof of Theorem \ref{thm: local gap} in Section 2, we shall simply assume that the shrinker in question has bounded curvature. In this way, all the techniques developed in \cite{CMZ21} can be applied to the canonical form defined as in \eqref{canonical form}. Let us recall some results in \cite{Bam20a} and \cite{CMZ21}.

\subsection{Conjugate heat kernel and Nash entropy}

Bamler \cite{Bam20a} studied the $W_1$-Wasserstein distance between conjugate heat kernels. For two probability measures $\mu$, $\nu$ in a metric space $(X,d)$, the $W_1$-Wasserstein distance is defined as 
\begin{align}\label{W1}
    \dist_{W_1}^X(\mu,\nu)=\sup_f\left(\int_X f\,d\mu-\int_X f\,d\nu\right),
\end{align}
where the infimum is taken over all bounded $1$-Lipschitz functions.

Now we consider a Ricci flow $(M,g_t)_{t\in I}$ with bounded curvature within each compact time interval. Let $\mu_t$ and $\nu_t$ be conjugate heat flows, namely,
$$\mu_t=u(\cdot,t)\,dg_t,\qquad \nu_t=v(\cdot,t)\,dg_t,$$
where $u(x,t)$ and $v(x,t)$ are positive solutions to the conjugate heat equation with unit integral. By applying a simple gradient estimate for the heat equation, Bamler \cite[Lemma 2.7]{Bam20a} shows that
$$t\to \dist_{W_1}^{g_t}(\mu_t,\nu_t)$$
is an increasing function in $t$.

In practice, we shall mostly consider conjugate heat kernels as conjugate heat flows, and denote by
$$d\nu_{x,t\,|\,s}=K(x,t\,|\,\cdot,s)\,dg_s$$
the conjugate heat kernel based at $(x,t)$. The Nash entropy based at $(x,t)$ is defined as
\begin{align*}
    \NN_{x,t}(\tau)=-\int_M\log K(x,t\,|\,\cdot,t-\tau)\,d\nu_{x,t\,|\,t-\tau}-\frac{n}{2}\log4\pi\tau-\frac{n}{2}.
\end{align*}

For any point $(x,t)\in M\times(-\infty,1)$ and any $s<t$, Bamler \cite{Bam20a} shows that there is always an $H_n$-center $(z,s)$ of $(x,t)$, such that $\nu_{x,t\,|\,s}$ is close to $\delta_z$, for instance, in the sense of $W_1$-Wasserstein distance. Precisely, we have (c.f. \cite[Corollary 3.8]{Bam20a})
\begin{align}\label{H_n-concentration}
    \dist_{W_1}^{g_s}\big(\nu_{x,t\,|\,s},\delta_z\big)\le \sqrt{H_n(t-s)},
\end{align}
where $H_n=\frac{(n-1)\pi^2}{2}+4$. It is clear from \cite{Bam20a} and \cite{CMZ21} %[local Sobolev] 
that the Nash entropy is strongly involved with the local geometry at the base point as well as at the $H_n$-center.

\begin{Theorem}[{\cite[Theorem 1.10]{CMZ21}%{local Sobolev}
}]
\label{thm: Nash depends on nu}
Assume that $[-r^2,0]\subseteq I$. Furthermore, assume that $R_{g_{-r^2}}\geq \Rmin$. Then, for any $x_0\in M$, any $H_n$-center $(z,-r^2)$ of $(x_0,0)$, and any $A\ge 8$, we have
\begin{align}\label{eq:nash depend on mu}
     \mu\left(B_{{-r^2}}\left(z,2A\sqrt{H_n}r\right), g_{-r^2},r^2\right)
    \le \NN_{x_0,0}(r^2)
    +  C(n,\Rmin r^2,A),
\end{align}
where
$$C(n,\Rmin r^2,A)=\tfrac{C_n}{A^2}e^{-\frac{A^2}{20}}+8\left(e^{-\frac{A^2}{20}}\cdot (n-2\Rmin r^2)+e^{-\frac{A^2}{40}}\cdot(n-2\Rmin r^2)^{\frac{1}{2}}\right),$$
and $C_n$ is a dimensional constant.
\end{Theorem}

Next, we recall the following $\varepsilon$-regularity theorem due to Bamler, which states that if on a Ricci flow the Nash entropy based at a point is small enough, then the curvature scale
cannot be too small.

\begin{Theorem}[{\cite[Theorem 10.2]{Bam20a}}]\label{epsilon regularity}
There exists a positive dimensional constant $\varepsilon(n)>0$ with the following property. Assume that $(x,t)\in M\times I$ and $[t-r^2,t]\subset I$. If $\NN_{x,t}(r^2)\ge -\varepsilon$, then $r_{\Rm}(x,t)\ge \varepsilon r$.
\end{Theorem}

\subsection{Another proof of Theorem \ref{thm: local gap}}

We would like to briefly summarize the proof of Theorem \ref{thm: local gap} presented in the section. The central idea is to apply Theorem \ref{thm: Nash depends on nu} properly in space-time of the canonical form defined in \eqref{canonical form}. Let us fix an arbitrary point $(x,0)$ in the canonical form. It is well-understood that as $\tau_i\to\infty$, the scaled conjugate heat kernel $\nu_{x,0\,|\,\tau_it}$ converges to $\tilde u_t\,dg_t$ (c.f.\cite{CZ10}), where $\tilde u_t$ is the ``singular conjugate heat kernel'' defined using the potential function in \eqref{standard CHK}. Since $\tilde u_t$ satisfies the gaussian estimates \eqref{Cao-Zhou}, and since the minimum point $o$ of the potential function is static in the canonical form, it is natural to assert that, whenever $t\ll-1$, $\nu_{x,0\,|\,t}$ and $\tilde u_t\,dg_t$ are close in the $W_1$-Wasserstein sense. Hence, for any $r\gg 1$, $(o,-r^2)$ is almost an $H_n$-center of $(x,0)$. 

On the other hand, since the canonical form moves by diffeomorphism, along which $o$ is static, one sees that, for all $r\gg 1$, the local geometry in $B_{g_{-r^2}}(o, \delta^{-1}r)$ is derived from the geometry in $B_{g}(o,\delta^{-1})$. Thus, $\mu\big(B_{g_{-r^2}}(o,\delta^{-1}r),g_{-r^2},r^2\big)\approx\mu\big(B_g(o,1/\delta),g,1\big)$. Combining the two observations above, we may apply Theorem \ref{thm: Nash depends on nu} to conclude $\NN_{x,0}(r^2)\approx 0$ for all $r\gg 1$, the conclusion then follows from Theorem \ref{epsilon regularity}.

Let $(M^n,g,f)$ be a Ricci shrinker with bounded curvature satisfying 
\begin{align}\label{smallness assumption}
    \mu\big(B(o,1/\delta),g,1\big)\ge -\delta,
\end{align}
where $\delta\in(0,1)$ is a positive dimensional constant to be determined and $o$ is the point where $f$ attains its minimum. We shall consider the canonical form of the shrinker $(M,g_t)_{t\in(-\infty,1)}$ as defined in \eqref{canonical form}. We also denote by 
$$d\mu_t=\tilde u_t\,dg_t,$$
where $\tilde u_t$ is defined in \eqref{standard CHK}, the canonical conjugate heat flow generated by the shrinker potential function. To begein with, we present two preparatory lemmas.

\begin{Lemma}[Continuous dependence of $\mu$-functional on the scale]\label{continuity}
Let $\Omega$ be a precompact domain on a smooth Riemannian manifold $(M^n,g)$ with positive scalar curvature. Then $\mu(\Omega,g,\tau)$ depends continuously on $\tau$.
\end{Lemma}
\begin{proof}
This result follows directly from the proof of \cite[Lemma 17, Proposition 5]{LW20}. First of all, since $\Omega$ is precompact and the scalar curvature is positive on $M$, it is obvious that on $\Omega$ there is a Sobolev inequality as follows.
\begin{align}\label{local Sobolev}
   % \left(\int u^{\frac{2n}{n-2}}\,dg\right)^{\frac{n-2}{2}}\le C_S\int\left(4|\nabla u|^2+Ru^2\right)\,dg,\qquad\text{ for all }\ u\in C_0^{0,1}(\Omega).
      \left(\int u^{\frac{2n}{n-2}}\,dg\right)^{\frac{n-2}{n}}\le C_S\int\left(4|\nabla u|^2+Ru^2\right)\,dg,\qquad\text{ for all }\ u\in C_0^{0,1}(\Omega).
\end{align}
The crux of the proof in \cite{LW20} is to apply the Sobolev inequality \eqref{local Sobolev} above to obtain an estimate for the energy part of the $\WW$-functional, namely,
$$\tau\int\left(4|\nabla u|^2+Ru^2\right)\, dg.$$
With the validity of \eqref{local Sobolev}, the rest of the proof follows word-by-word from \cite[Lemma 17, Proposition 5]{LW20}.
\end{proof}

\begin{Lemma}\label{standard CHK concentration}
Under the assumption \eqref{smallness assumption}, we have
\begin{align}
    \dist_{W_1}^{g_t}(\tilde\mu_t,\delta_o)\le C(n)\sqrt{\tau_t}, \qquad\text{ for all }\ t\in(-\infty,1).
\end{align}
\end{Lemma}

\begin{proof}
Since $o$ is static under the flow $\phi_t$ defined in \eqref{canonical form}, we have
\begin{align*}
     \dist_{W_1}^{g_t}(\tilde\mu_t,\delta_o)=\dist_{W_1}^{\tau_t\phi_t^*g}(\phi_t^*\tilde\mu_0,\phi_t^*\delta_o)=\sqrt{\tau_t}\dist_{W_1}^{g}(\tilde\mu_0,\delta_o),
\end{align*}
where the last equation above is from the fact that $h$ is $1$-Lipschitz with respect to $g$ if and only if $\sqrt{\tau_t}h\circ\phi_t$ is $1$-Lipschitz with respect to $g_t$. Thus, if suffices to obtain an estimate like \begin{align*}
    \dist_{W_1}^g(\tilde \mu_0,\delta_o)\le C(n).
\end{align*}

By the assumption \eqref{smallness assumption} and Lemma \ref{lem: rough local mu bound}, we have a lower bound of $\mu_g$ depending only on the dimension. Thus, \eqref{Cao-Zhou} implies that
\begin{align}\label{New Cao-Zhou}
   \frac{1}{4}\left(\dist_g(x,o)-C(n)\right)_+^2\le \tilde f(x)\le \frac{1}{4}\left(\dist_g(x,o)+C(n)\right)^2.
\end{align}
On the other hand, for any bounded $1$-Lipschitz function $h$ (with respect to $g$), we have
\begin{align*}
    \int_M h\,d\tilde\mu_0-\int_M h\,d\delta_o&=\int_M\big(h(x)-h(o)\big)(4\pi)^{-\frac{n}{2}}e^{-\tilde f(x)}\,dg(x)
    \\
    &\le \int_M\dist_g(x,o)(4\pi)^{-\frac{n}{2}}e^{-\tilde f(x)}\,dg(x).
\end{align*}
It follows immediately from \eqref{MW14}, \eqref{W1}, and \eqref{New Cao-Zhou} that
\begin{align*}
    \dist_{W_1}^{g}\big(\tilde\mu_0,\delta_o\big)\le \int_M\dist_g(x,o)(4\pi)^{-\frac{n}{2}}e^{-\tilde f(x)}\,dg(x)\le C(n).
\end{align*}
We have finished the proof of the Lemma.
\end{proof}

\begin{proof}[Another proof of Theorem \ref{thm: local gap}]
Let $(M,g_t)_{t\in(-\infty,1)}$ be the canonical form defined in \eqref{canonical form} of a shrinker with bounded curvature satisfying \eqref{smallness assumption}. Let $o$ be a %the 
minimum point of $f$.

In order to deal with the diffeomorphism in the definition of the canonical form, we shall consider sub-level-sets of $f_t=f\circ\phi_t$ instead of geodesic balls. Define
\begin{align}\label{def of Dt}
    D_t(r):=\left\{f_t\le \tfrac{1}{4}r^2\right\},\qquad  t\in(-\infty,1),\ \ r^2>0.
\end{align}
Then, by \eqref{Cao-Zhou}, it is easy to see that
\begin{align*}
    B_{g_t}\Big(o,(r-\sqrt{2n}\,)\sqrt{\tau_t}\Big)\subset D_t(r)\subset B_{g_t}\Big(o,(r+5n)\sqrt{\tau_t}\Big),\qquad\text{ for all }\ t\in(-\infty,1)\ \text{ and }\ r>\sqrt{2n}.
\end{align*}
Thus, if we take $\delta\le\overline{\delta}(n)\ll 1$, then \eqref{smallness assumption} can be rewritten as 
\begin{align}\label{new smallness assumption}
    \mu\big(D_0(1/\delta),g_0,1\big)\ge -\delta.
\end{align}
In addition, 
\begin{align*}
    \phi_t: \Big(D_t(r),g_t\Big)&\to \Big(D_0(r),\tau_t g\Big)
\end{align*}
is obviously an isometry, which implies that
\begin{align}\label{equivalence of mu}
    \mu\big(D_t(1/\delta),g_t,s\tau_t\big)\equiv\mu\big(D_0(1/\delta),g,s\big),\qquad\text{ for all }\ t\in(-\infty,1) \text{ and } s>0.
\end{align}

Next, we fix an arbitrary point $x\in M$. Let $r\ge 1$ and let $(z,-r^2)$ be an $H_n$-center of $(x,0)$.
By the monotonicity of the $W_1$-Wasserstein distance and Lemma \ref{standard CHK concentration}, we may estimate
\begin{align*}
    \dist_{g_{-r^2}}(z,o)&=\dist^{g_{-r^2}}_{W_1}(\delta_z,\delta_o)
    \\
    &\le \dist^{g_{-r^2}}_{W_1}(\delta_z,\nu_{x,0\,|\,-r^2})+\dist^{g_{-r^2}}_{W_1}(\tilde\mu_{-r^2},\nu_{x,0\,|\,-r^2})+\dist^{g_{-r^2}}_{W_1}(\tilde\mu_{-r^2},\delta_o)
    \\
    &\le \sqrt{H_n}r+\dist_{W_1}^{g_0}(\tilde\mu_0,\delta_x)+C(n)\sqrt{\tau_{-r^2}}
    \\
    &\le \left(C(n)+\frac{\dist_{W_1}^{g_0}(\tilde\mu_0,\delta_x)}{\sqrt{\tau_{-r^2}}}\right)\sqrt{\tau_{-r^2}}.
\end{align*}
By the same argument as in the proof of Lemma \ref{standard CHK concentration}, $\dist_{W_1}^{g_0}(\tilde\mu_0,\delta_x)$ is finite. Thus, taking $r\ge\overline{r}(n,x)$ such that
$$\frac{\dist_{W_1}^{g_0}(\tilde\mu_0,\delta_x)}{\sqrt{\tau_{-r^2}}}\le 1,$$
we have $z\in D_{-r^2}\big(C(n)+1\big)$, and hence, taking $\delta\le \overline{\delta}(n)$,
$$B_{g_{-r^2}}\big(z,\tfrac{1}{2}\delta^{-1}r\big)\subset D_{-r^2}(1/\delta).$$

Finally, applying Theorem \ref{thm: Nash depends on nu} at $(x,0)$ with $2A\sqrt{H_n}=\frac{1}{2}\delta^{-1}$ and $\Rmin=0$ leads to
\begin{align*}
    \NN_{x,0}(r^2)&\ge \mu\left(B_{g_{-r^2}}\left(z,\tfrac{1}{2}\delta^{-1}r\right),g_{-r^2},r^2\right)-C_n\delta^2 e^{-c_n\delta^{-2}}
    \\
    &\ge \mu\left(D_{-r^2}(1/\delta),g_{-r^2},\tfrac{r^2}{1+r^2}\tau_{-r^2}\right)-C_n\delta^2 e^{-c_n\delta^{-2}}
    \\
    &=\mu\left(D_0(1/\delta),g,\tfrac{r^2}{1+r^2}\right)-C_n\delta^2 e^{-c_n\delta^{-2}}
    \\
    &\ge \mu\left(D_0(1/\delta),g,1\right)-\delta-C_n\delta^2 e^{-c_n\delta^{-2}}
    \\
    &\ge -2\delta-C_n\delta^2 e^{-c_n\delta^{-2}},
\end{align*}
if we take $r\ge \overline{r}(n,x,\delta)$, where we have also applied \eqref{new smallness assumption}, \eqref{equivalence of mu}, and Lemma \ref{continuity}. Therefore, if $\delta\le\overline{\delta}(n)$ such that 
$$2\delta+C_n\delta^2 e^{-c_n\delta^{-2}}\le \varepsilon(n),$$
 where $\varepsilon(n)$ is the constant in Theorem \ref{epsilon regularity}, then we have
 $$\NN_{x,0}(r^2)\ge -\varepsilon(n),\qquad \text{ for all }\ r\ge \overline{r}(n,x,\delta).$$
We may conclude by Theorem \ref{epsilon regularity} that $r_{\Rm}(x,0)=\infty$. Hence the shrinker in question must be flat.
\end{proof}

\section{Proof of Theorem \ref{rm gap}}

In this section, we prove Theorem \ref{rm gap}. We first provide a local injectivity radius estimate at the minimum point $o$ of the potential function. Then, we apply a result by Hamilton \cite{Ham95} to show an almost Euclidean isoperimetric inequality on a ball of large radius centered at $o$. Theorem \ref{rm gap} then follows from Corollary \ref{coro}. %Theorem \ref{thm: local gap}. 
Indeed, it can be seen from the proof of the Gromoll-Meyer theorem \cite{GM69} that,
 on a complete Riemannian manifold with sectional curvature bounded from above by a positive constant $K$, if there is a proper $C^2$ strictly convex function, then the injectivity radius of $M$ is bounded from below by $\frac{\pi}{\sqrt{K}}$ (see also \cite[Theorem B.65]{CD04}); a local versinon of this statement is sufficient for our purpose.

\begin{Lemma}\label{inj estimate} Let $(M^n,g)$ be an $n$-dimensional complete Riemannian manifold. %There exists a dimensional constant $r_0$ such that 
Let $r>0$, $p_0\in M$, and $F:B(p_0,r)\to\mathbb{R}$ be a strictly convex function. Assume that
the sectional curvature is bounded from above by $r^{-2}$ on $B_g(p_0, r)$, and that $\{F\leq s\}\subseteq B_g(p_0, r/2)$ for some $s\in \mathbb{R}$. Then for any $x\in \{F\leq s\}$, the injectivity radius at $x$ satisfies
\[
\operatorname{inj}_g(x)\geq \frac{r}{2}.
\]
\end{Lemma}

\begin{proof}Let $A:=\{F\leq s\}$. By the Rauch comparison theorem, we have that the conjugate radius at each $x\in A$ is bounded from below by $\min\{\pi r, r/2\}=r/2$ \cite[Theorem B.20]{CD04}. Suppose by contradiction that the conclusion of the lemma fails. Then
\[
\iota_0:=\inf\{\operatorname{inj}_g(y): y\in A\}<\frac{r}{2}.
\]
By the compactness of $A$ and the conjugate radius bound,  we have $\iota_0>0$. Let $\{p_m\}$ be a sequence in $A$ such that  $\text{inj}_g(p_m)\to \iota_0$ as $m\to \infty$. By passing to subseqence and by a result of Klingenberg \cite[Lemma 1]{K59} (see also \cite[Prop 2.12, Ch.13]{doC92}), $p_m\to p_{\infty}\in A$ and there is a sequence of unit-speed geodesics $\gamma_m:[0,2I_m]\to M$, where $I_m=\text{inj}_g(p_m)$, with $\gamma_m(0)=\gamma_m(2I_m)=p_m$ such that $q_m:=\gamma_m(I_m)$ is a point on the cut locus of $p_m$ that realizes the distance between $p_m$ and its cut locus. By the compactness of $A$, we may further suppose that $q_m\to q_\infty$ and $\gamma_m\to \gamma_\infty$, where $\gamma_\infty:[0,2\iota_0]\to M$ satisfies $\gamma_\infty(0)=\gamma_\infty(2\iota_0)=p_\infty$, and $q_\infty=\gamma_{\infty}(\iota_0)$ is a point on the cut locus of $p_\infty$ which realizes the distance between $p_\infty$ and its cut locus. Note that in general $\gamma_m$ may have corner at $p_m$, i.e. $\gamma_m'(0)\neq \gamma'_m(2I_m)$. We will prove that this doesn't happen in the limit $\gamma_\infty$. As $p_\infty\in A$ and $\iota_0<r/2$, we have that $\gamma_\infty$ is contained in $B_g(p_0,r)$. We then show that $\gamma_\infty$ lies inside $A$. The function $F\circ\gamma_\infty$ has positive second derivative, and hence, by the property of convex functions, we have
\[
F\circ\gamma_\infty(t)\le F\circ\gamma_\infty(0)=F\circ\gamma_\infty(2\iota_0)=F(p_{\infty})\leq s.
\]
As a consequence, $\gamma_\infty$ and, in particular, $q_{\infty}$ lies in $A$. By the minimality of $\iota_0$, we may apply the same argument as in \cite[Prop 2.12, Ch.13]{doC92} to $q_{\infty}$ to see that $\gamma'_\infty(0)=\gamma'_\infty(2\iota_0)$, and thus $\gamma_\infty$ is a closed smooth geodesic loop without any corner. This implies that the function $F\circ\gamma_\infty(t)$ can be extended to a smooth nonconstant $2\iota_0$ periodic function on $\mathbb{R}$, and thus attains its maximum value, which is impossible for a strictly convex function on $\mathbb{R}$. 
\end{proof}

To prove the gap theorem, we also need the following estimate of the metric in its geodesic coordinates by Hamilton \cite[Theorems 4.9 and 4.10]{Ham95}.

\begin{Theorem}\label{geod est}\emph{\cite[Theorems 4.9 and 4.10]{Ham95} }
There exist dimensional constants $c$ and $C_0$ with the following property. Let $(M,g)$ be a Riemannian manifold. Assume  that $|{\Rm}|_g\le B_0$ on $B_g(p,s)$ for some constant $B_0>0$ and that the injectivity radius at $p\in M$ is no less than $s>0$, where $s\le c/\sqrt{B_0}$. Then it holds that 
\[
|g_E-g|_g\leq 4B_0C_0s^2 \ \ \text{ on $B_g(p,s)$},
%|\delta-g|_g\leq 4B_0C_0s^2 \ \ \text{ on $B_g(p,s)$},
\]
where $g_E$ %$\delta$ 
is the Euclidean metric in the geodesic coordinates at $p$, namely, the local pull-back of the Euclidean metric $g_p$ on $T_pM$ by  $\exp_p^{-1}$.
\end{Theorem}

\begin{proof}[Proof of Theorem \ref{rm gap}]
   Let us consider a Ricci shrinker $(M^n,g,f)$ satisfying the condition prescribed in the statement of Theorem \ref{rm gap}. Let $\ve>0$ be some dimensional constant to be determined. By the assumption on the curvature radius, we have
\begin{align*}
    |{\Rm}|\le \ve^2\ \ \ \text{ on $\ B_g(o,1/\ve)$}.
\end{align*}
Consequently, the smallness of the Ricci tensor implies that, if $\varepsilon\le\overline{\varepsilon}(n)\ll 1$, then 
\[
0<\frac{1}{4}g\le \left(\frac{1}{2}-c_n\ve^2\right)g\le \nabla^2 f\le  \left(\frac{1}{2}+c_n\ve^2\right)g\le \frac{3}{4}g,\ \ \ \text{ on $\ B_g(o, 1/\ve)$.}
\]
Hence $f$ is striclty convex on $B_g(o, \ve^{-1})$. Moreover, by \eqref{Cao-Zhou}, for any $r\ge \sqrt{2n}$, we have
\[
B_g(o,r-\sqrt{2n}\, )\subseteq D_0(r)\subset B_g(o,r+5n),
\]
where $D_0(r)$ is the sub-level set of $f$ defined as in \eqref{def of Dt}. Hence, letting $r=\frac{1}{4\ve}$, if $\ve\le \overline{\ve}(n)\ll 1$, then $D_0\big(1/(4\varepsilon)\big)\subseteq B_g\big(o,1/(2\ve)\big)$. By Lemma \ref{inj estimate}, $\operatorname{inj}_g(o)\ge 1/(8\ve)$. Applying Theorem \ref{geod est} to $B_g\big(o,s\big)$ with $B_0=\ve^2$ and $s=c'/\ve$, where $c'\in\left(0,\min\{\frac{1}{8},c\}\right]$ is to be fixed, we have
\begin{equation}\label{g vs d}
   % |\delta-g|_g\leq 4C_0c'^2, \ \ \ \text{ on $\ B_g(o, c'/\ve)$}.
    |g_E-g|_g\leq 4C_0c'^2, \ \ \ \text{ on $\ B_g(o, c'/\ve)$},
\end{equation}
where $g_E$ is the local pull-back of the Euclidean metric $g_o$ on $T_oM$ by $\exp_o^{-1}$.
%Hence for any $\delta$, there is a small $c'>0$ depending only on $n$ and $\delta$ such that the almost Euclidean isoperimetric inequality holds on $B_g(o, \frac{c'}{\ve})$. 

Next, let $\delta=\delta(n)>0$ be the small constant in Corollary \ref{coro}.
%, and then choose a $0<\delta'<\delta$ depending only on dimension such that
%\[
%n\log(1-\delta')>-\delta.
%\]
In view of \eqref{g vs d}, if we take $c'\le\overline{c'}(n,\delta)\ll 1$, then the almost Euclidean isoperimetric inequality holds on $B_g(o, c'/\ve)$, namely, 
\begin{equation}\label{isoper}
%\left(\operatorname{Area}_{g}(\partial \Omega)\right)^{n} \geq(1-\delta') n^n\omega_{n}\left(\operatorname{Vol}_{g}(\Omega)\right)^{n-1}\qquad\text{ for any regular } \Omega\subset B_g(o, c'/\ve).
\left(\operatorname{Area}_{g}(\partial \Omega)\right)^{n} \geq(1-\delta) n^n\omega_{n}\left(\operatorname{Vol}_{g}(\Omega)\right)^{n-1}\qquad\text{ for any regular } \Omega\subset B_g(o, c'/\ve).
\end{equation}
Indeed, to see \eqref{isoper}, 
%we let $\text{exp}_o: B_{g_E}(0, c'/\ve)\longrightarrow B_g(o, c'/\ve)$ be the exponential map at $o$, where $B_{g_E}(0, c'/\ve)$ is the Euclidean Ball in the tangent space centered at $0$ with radius $c'/\ve$. Abusing the notation a bit, we shall not distinguish $g$ and $\text{exp}_o^*\,g$.
we may apply \eqref{g vs d} to obtain
\[
%(1-c_nc'^2)\delta_{ij}\le g_{ij}\le (1+c_nc'^2)\delta_{ij},
(1-c_nc'^2)g_E\le g\le (1+c_nc'^2)g_E, \ \ \ \text{ on $\ B_g(o, c'/\ve)$}.
\]
where $c_n$ is a dimensional constant. Hence for any regular $\Omega\subset B_g(o, c'/\ve)$, it holds that
\[
 (1-c_nc'^2)^{\frac{n-1}{2}}\operatorname{Area}_{g_E}\left(\partial \Omega\right)\le \operatorname{Area}_{g}\left(\partial \Omega\right)\le (1+c_nc'^2)^{\frac{n-1}{2}}\operatorname{Area}_{g_E}\left(\partial \Omega\right)
%  (1-c_nc'^2)^{\frac{n-1}{2}}\operatorname{Area}_{\delta}\left(\text{exp}_o^{-1}\left(\partial \Omega\right)\right)\le \operatorname{Area}_{g}\left(\text{exp}_o^{-1}\left(\partial \Omega\right)\right)\le (1+c_nc'^2)^{\frac{n-1}{2}}\operatorname{Area}_{\delta}\left(\text{exp}_o^{-1}\left(\partial \Omega\right)\right)
\]
\[
(1-c_nc'^2)^{\frac{n}{2}} \operatorname{Vol}_{g_E}\left(\Omega\right) \le \operatorname{Vol}_{g}\left(\Omega\right)\le (1+c_nc'^2)^{\frac{n}{2}} \operatorname{Vol}_{g_E}\left(\Omega\right).
%(1-c_nc'^2)^{\frac{n}{2}} \operatorname{Vol}_{\delta}\left(\text{exp}_o^{-1}\left(\Omega\right)\right) \le \operatorname{Vol}_{g}\left(\text{exp}_o^{-1}\left(\Omega\right)\right)\le (1+c_nc'^2)^{\frac{n}{2}} \operatorname{Vol}_{\delta}\left(\text{exp}_o^{-1}\left(\Omega\right)\right).
\]
Thus, if we take $0<c'\le \overline{c'}(n,\delta)\ll 1$, then
\begin{eqnarray*}
\frac{\left(\operatorname{Area}_{g}(\partial \Omega)\right)^{n}}{\left(\operatorname{Vol}_{g}(\Omega)\right)^{n-1}}&\ge& \frac{ (1-c_nc'^2)^{\frac{n(n-1)}{2}}}{(1+c_nc'^2)^{\frac{n(n-1)}{2}}}\frac{\left[ \operatorname{Area}_{g_E}\left(\partial \Omega\right)\right]^{n}}{\left[\operatorname{Vol}_{g_E}\left(\Omega\right)\right]^{n-1}} \\
%&\ge& \frac{ (1-c_nc'^2)^{\frac{n(n-1)}{2}}}{(1+c_nc'^2)^{\frac{n(n-1)}{2}}}\frac{\left[ \operatorname{Area}_{\delta}\left(\text{exp}_o^{-1}\left(\partial \Omega\right)\right)\right]^{n}}{\left[\operatorname{Vol}_{\delta}\left(\text{exp}_o^{-1}\left(\Omega\right)\right)\right]^{n-1}} \\
&\ge&\frac{ (1-c_nc'^2)^{\frac{n(n-1)}{2}}}{(1+c_nc'^2)^{\frac{n(n-1)}{2}}}n^n\omega_{n}
\\
&\ge& (1-\delta)n^n\omega_{n},
%\ge (1-\delta')n^n\omega_{n},
\end{eqnarray*}
where we also used the Euclidean isoperimetric inequality in the second line; \eqref{isoper} is proved. Lastly, we pick $\ve\le\overline{\ve}(n,c',\delta)$ such that $\frac{c'}{\ve}>\frac{1}{\delta}$; note that the upper bound of $\ve$ depend only on $n$, since our choice of $c'$ and $\delta$ also depend only on $n$. %$\frac{c'}{\ve}>\frac{1}{\delta'}$. 
%It follows from \eqref{isoper} and Lemma \ref{nu estimate} that
%\[
% \mu(B(o,1/\delta), g,1) \ge \nu(B(o,1/\delta'),g,1/\delta'^2)\ge n\log(1-\delta')>-\delta.
%\]
This shows that the lower bound of the isoperimetric constant of $B_g(o,1/\delta)$ is $\delta$-close to that of the Euclidean space.
By Corollary \ref{coro}, 
%By Theorem \ref{thm: local gap}, 
the shrinker is the flat gaussian soliton. This completes the proof of Theorem \ref{rm gap}.  
\end{proof}

%\bibliography{bibliography}{}
\bibliographystyle{amsalpha}

\newcommand{\alphalchar}[1]{$^{#1}$}
\providecommand{\bysame}{\leavevmode\hbox to3em{\hrulefill}\thinspace}
\providecommand{\MR}{\relax\ifhmode\unskip\space\fi MR }
% \MRhref is called by the amsart/book/proc definition of \MR.
\providecommand{\MRhref}[2]{%
  \href{http://www.ams.org/mathscinet-getitem?mr=#1}{#2}
}
\providecommand{\href}[2]{#2}

\noindent Department of Mathematics, University of California, San Diego, CA 92093, USA
\\ E-mail address: \verb"pachan@ucsd.edu "
\\

%\noindent Department of Mathematics, University of California, San Diego, CA 92093, USA
%\\ E-mail address: \verb"zim022@ucsd.edu"
%\\

\noindent Department of Mathematics, Rutgers University, Piscataway, NJ 08854, USA
\\ E-mail address: \verb"zilu.ma@rutgers.edu"
\\

\noindent School of Mathematical Sciences, Shanghai Jiao Tong University, Shanghai, 200240, China
\\ E-mail address: \verb"sunzhang91@sjtu.edu.cn"

\end{document}